\documentclass[11pt]{amsart}

\usepackage{amssymb,amsmath,amsfonts,amsthm,color}
\usepackage{bbold}   

\newtheorem{remark}{Remark}[section]
\newtheorem{theorem}{Theorem}[section]
\newtheorem{corollary}{Corollary}[section]
\newtheorem{lemma}{Lemma}[section]
\def\uno{\mathbb 1}
\newcommand{\fro}{\mathrm{F}} 
\renewcommand{\t}{\mathrm{T}} 
\newcommand{\Tr}{\mathrm{tr}} 

\numberwithin{equation}{section}

\begin{document} 

\title[Localization of dominant eigenpairs 
and planted communities]{Localization of dominant eigenpairs 
and planted communities 
by means of Frobenius inner products}

\author{Dario Fasino}

\author{Francesco Tudisco}

\maketitle 

\centerline{\em \normalsize{Dedicated to the memory of Professor Miroslav Fiedler}}









\begin{abstract}
We propose a new localization result for the leading eigenvalue and eigenvector of a symmetric matrix $A$. The result exploits the  Frobenius inner product between $A$ and a given rank-one landmark matrix $X$. Different choices for $X$ may be used, depending upon the problem under investigation. In particular, we show that the choice where $X$ is the all-ones matrix allows to estimate the signature of the leading eigenvector of $A$, generalizing previous results on Perron-Frobenius properties of matrices with some negative entries. As another  application we consider the problem of 
community detection in graphs and networks. The problem is solved by means of modularity-based spectral techniques, following the ideas pioneered by Miroslav Fiedler in mid 70s. We show that a suitable choice of $X$ can be used to provide new quality guarantees of those techniques, when the network follows a stochastic block model.
\end{abstract}

\smallskip
\noindent \textbf{Keywords. }  Dominant eigenpairs, cones of matrices,
 spectral method, community detection.  \textit{2010 MSC:} 15A18,  
 15B48  %




\section{Introduction}

Consider the following result, found in \cite{tarazaga2001}:

\begin{theorem}   \label{thm:tarazaga}
If a symmetric $n\times n$ matrix $A$ satisfies 
\begin{equation}   \label{eq:tarazaga}
   \uno^\t A\uno \geq \sqrt{(n-1)^2+1}\,\|A\|_\fro 
\end{equation}
then  its spectral radius is a simple eigenvalue and a corresponding eigenvector is nonnegative.
\end{theorem}

The symbol $\uno$ above stands for a vector
whose entries are $1$'s.
This theorem, whose original proof relies upon 
studying specific convex cones of nonnegative matrices and the solution of certain 
linear programming problems, shows various intriguing aspects.
For example, it reveals that a result which is typical 
of the Perron-Frobenius theory of nonnegative matrices 
can be valid for matrices having some negative entries.
Moreover, it can establish a localization property 
of the dominant eigenvector with respect to the central 
ray of the positive orthant, without involving the
spectral separation of the associated eigenvalue.
On the other hand, the authors of \cite{tarazaga2001}
missed other interesting consequences. For example, in the hypotheses
of the aforementioned theorem, the rightmost eigenvalue
of $A$ is larger than but quite close to $r/n$ where $r$ denotes
the right hand side of \eqref{eq:tarazaga},
as we shall prove in the following.

The purpose of this paper is to shed light on some imaginable generalizations
of Theorem \ref{thm:tarazaga}, along with a few more consequences and
a possible  application in network analysis. 
It is well known that in many occasions the analysis of graphs and networks
interacts profitably with matrix theory, see \cite{FiedlerRev}.
Also in this circumstance we discovered an unexpected contact 
between these subjects.

The paper is organized as follows. Before concluding this introduction,
we collect hereafter some notation and preliminary results. In Section \ref{sec:localization} we state, prove and discuss our new localization result for the leading eigenvalue and an associated eigenvector of a symmetric matrix $A$. The result is given in terms of  the  angle between $A$ and a given non-zero rank-one matrix $X=xx^\t$, measured in terms of the Frobenius inner product between $A$ and $X$. The generality of the landmark matrix $X$ allows to apply the theorem in various contexts. In Section \ref{sec:signature} we observe that the choice $X=\uno \uno^\t$  can be used to obtain an improved version of  Theorem \ref{thm:tarazaga}, with an alternative proof and further details 
on the size of the leading eigenvalue of $A$, and on the signature of the associated leading eigenvector. 
In the last section we apply our main results to the analysis of a simple spectral method in community detection. We consider the planted partition model,
which is a widespread benkmark in community detection, and 
estimate the fraction of correctly classified vertices that is almost certainly 
attained in large graphs.

\subsection{The Frobenius inner product and matrix norm}

The space 
of $n\times n$ real matrices is naturally endowed with
the Frobenius inner product $\langle A,B\rangle = \Tr(AB^\t)$
and the associate matrix norm $\|A\|_\fro = \langle A,A\rangle^\frac12$.
Correspondingly, the angle between $A$ and $B$
is defined 
by the respective cosine as follows:
\begin{equation}   \label{eq:costheta}
   \cos(A,B) = \frac{\mathrm{tr}(A B^\t)}{\|A\|_\fro\|B\|_\fro} .
\end{equation}
For any fixed matrix $B$ let $\mathcal{P}_B(A)$ denote 
the orthogonal projection of a generic matrix $A$ onto the linear space spanned by $B$. We have the explicit expression
$\mathcal P_B(A) = \tau(A) B$ where
$$
   \tau(A) = \frac{\Tr(AB^\t)}{\|B\|_\fro^2} . 
$$
Indeed, consider the decomposition $A = \tau(A)B + Z$.
The two terms of that decomposition are orthogonal with respect to the Frobenius inner product. In fact, 
$$
   \langle A - \mathcal{P}_B(A),B\rangle = 
   \langle A ,B\rangle - 
   \langle \mathcal{P}_B(A),B\rangle = 
   \langle A ,B\rangle - 
   \tau(A) \|B\|_\fro^2 = 0 , 
$$
and $Z$ is the residual of the projection.

\begin{remark}  \label{rem:rk1}
When $B = vv^\t$ is a symmetric, positive semidefinite rank-one matrix, 
equation \eqref{eq:costheta} becomes
\begin{equation}   \label{eq:cosrk1}
   \cos(A,vv^\t) = \frac{v^\t Av}{v^\t v\|A\|_\fro} .
\end{equation}
In that case we also have $\tau(A) = v^\t Av/(v^\t v)^2$ and
the only nontrivial eigenvalue of $\mathcal P_B(A)$ 
is $v^\t Av/v^\t v$.
\end{remark}

We conclude this introduction by recalling a few useful results.
If $A$ is symmetric, its eigenvalues are considered
in nonincreasing order, $\lambda_1(A) \geq \lambda_2(A) \geq \ldots$ and we have the formula $\|A\|_\fro^2 = \sum_{i=1}^n \lambda_i(A)^2$.
In what follows we will make use of a special case of the Hoffman-Wielandt theorem, see e.g.,
\cite[Thm.\ 9.21]{FiedlerBook}:

\begin{theorem}   \label{thm:HW}
Let $A$ and $B$ be symmetric $n\times n$ matrices, then 
$$
   \sum_{i=1}^n (\lambda_i(A)-\lambda_i(B))^2 
   \leq \|A-B\|_\fro^2 .
$$
\end{theorem}

\section{Localization of a dominant eigenpair} \label{sec:localization}  

Hereafter we prove that, if the angle between 
a symmetric matrix $A$ and a symmetric rank-one matrix is 
sufficiently small then $A$ has one simple, dominant eigenvalue, 
and the associated spectral projector is close to a multiple
of the rank-one matrix. 
If $u$ and $v$ are vectors, the notation $\cos(u,v)$ has the obvious meaning given by the Euclidean inner product.

\begin{theorem}   \label{thm:main}
Let $A$ be symmetric and let $X = xx^\t$ be a nonzero rank-one matrix.
Moreover, let $c = \cos(A,X) > 0$ and $s = \sqrt{1-c^2}$. Then:
\begin{enumerate}
\item
$\lambda_1(A) \geq\mu$ where $\mu = c\|A\|_\fro$.
In particular, if $c > 1/\sqrt{2}$ then $\lambda_1(A)$ is simple
and dominant;
\item
$|\lambda_1(A) - \mu|/|\lambda_1(A)| \leq s$;
\item
let $v_1$ be an eigenvector of $\lambda_1(A)$. 
If $c^2 \geq \frac12$ then $\cos(v_1,x)^2 \geq \xi$
where $\xi\in[\frac12,1]$ is the largest root of
$c^2 = 2\xi^2 - 2\xi + 1$. 
\end{enumerate}
\end{theorem}

\begin{proof}
Firstly, note the equivalent formulas
$$
   \mu = \cos(A,X) \|A\|_\fro = \frac{x^\t Ax}{x^\t x} . 
$$
Hence, the inequality $\lambda_1(A) \geq \mu$ is due to the variational characterization
of the largest eigenvalue of a symmetric matrix, see e.g.,
\cite[Thm.\ 2.30]{FiedlerBook}.

Let $Z = A - \mathcal{P}_X (A)$.
By simple trigonometry,
$\|Z\|_\fro = s \|A\|_\fro$.
Owing to Remark \ref{rem:rk1},
the matrix $\mathcal{P}_X (A)$ has only one (positive) eigenvalue, which is equal to
$\mu$. 
From Theorem \ref{thm:HW} we have
\begin{align*}
   s^2 \sum_{i=1}^n\lambda_i(A)^2 = 
   s^2 \|A\|_\fro^2 = \|Z\|_\fro^2 
   & = \| A - \mathcal{P}_X (A)\|_\fro^2 \\
   & \geq (\lambda_1(A) - \mu)^2 + \sum_{i=2}^n\lambda_i(A)^2 .
\end{align*}
Consequently, $s^2\lambda_1(A)^2 \geq (1-s^2)\sum_{i=2}^n\lambda_i(A)^2$
and
if $c> 1/\sqrt 2$ then $1-s^2> 1/2$ and we obtain $\lambda_1(A)^2>\sum_{i\neq 1}\lambda_i(A)^2$, completing the proof of the first claim.
We observe in passing that the last inequality can be expressed a follows: 
The numbers $\lambda_i(A)^2$ fulfil the strict polygonal inequality.
In fact, Fiedler used to call a polygonal inequality the relation
$2\max_i \alpha_i \geq \sum_i \alpha_i$ among nonnegative numbers $\alpha_i$.

\medskip
Rearranging terms we also get
$$
   (\lambda_1(A) - \mu)^2 \leq 
   s^2 \lambda_1(A)^2 - (1 - s^2) \sum_{i=2}^n\lambda_i(A)^2
   \leq 
   s^2 \lambda_1(A)^2 ,
$$
which proves the second claim.
Finally, let 
$$ 
   A = \sum_{i=1}^n \lambda_i v_iv_i^\t 
$$ 
be the spectral decomposition of $A$.
For notational simplicity, let $c_i = v_i^\t x/\|x\|$
and $\mu_i = \lambda_i/\|A\|_\fro$.
Note that $c_i$ is the cosine of the angle between 
the vectors $v_i$ and $x$.
With this auxiliary notation we obtain
$$
   \cos(A,X) 
   = \frac{\sum_{i=1}^n \lambda_i (v_i^\t x)^2}
   {x^\t x\,\|A\|_\fro}
   = \sum_{i=1}^n \mu_i c_i^2 .
$$
Owing to the equations
$$ 
   \sum_{i=1}^n \mu_i^2 = \sum_{i=1}^n c_i^2 = 1 
$$ 
and the inequalities $\mu_1 \geq \mu_2 \geq \ldots \geq \mu_n$, we have
\begin{equation}  \label{eq:bound}
   c = \cos(A,X) \leq \mu_1 c_1^2 + \mu_2 (1-c_1^2) 
   \leq \mu_1 c_1^2 + \sqrt{1-\mu_1^2}(1-c_1^2) .
\end{equation}
Arguing by contradiction, suppose $c_1^2 < \xi$.
By hypothesis and the first claim, we have $\mu_1^2 \geq c^2 \geq \frac12$. Consequently $\sqrt{1-\mu_1^2} \leq \mu_1$
and
$$
   \mu_1 c_1^2 + \sqrt{1-\mu_1^2}(1-c_1^2) <
   \mu_1 \xi + \sqrt{1-\mu_1^2}(1-\xi) .
$$
Consider the function
$f(\mu) = \mu \xi + \sqrt{1-\mu^2}(1-\xi)$. Simple computations prove that
$$
   \max_{0\leq \mu\leq 1} f(\mu) = \sqrt{2\xi^2 - 2\xi + 1} .
$$
Finally, 
$$
   \mu_1 c_1^2 + \sqrt{1-\mu_1^2}(1-c_1^2) <
   f(\mu_1) \leq \sqrt{2\xi^2 - 2\xi + 1} = c ,
$$
which contradicts \eqref{eq:bound}. Hence we must have $c_1^2 \geq \xi$.
\end{proof}

The foregoing theorem can be used as a localization tool for the 
extreme eigenvalues of $A$ and the corresponding eigenvectors. 
In fact, the quantity $\cos(A,xx^\t)$ is maximized when $x$
is an eigenvector associated to the rightmost eigenvalue of $A$. On the other hand, if $\cos(A,xx^\t) < 0$ 
we obtain immediately a corresponding result for the 
leftmost eigenvalue, by replacing $A$ with $-A$.  
It is worth noting that, unlike in  
Perron-Frobenius theory, Theorem \ref{thm:main} makes no assumption about
positivity of $A$. Furthermore, differently from 
classical results in spectral perturbation theory, the estimate
on $\cos(v_1,x)$ does not depend on the spectral separation 
of $\lambda_1(A)$.

\medskip

The following examples show that the hypotheses of Theorem \ref{thm:main} cannot be weakened in general.
Indeed, let $n$ be an even integer, let $x = (y,y)^\t$
where $y\neq 0$ is a vector with $n/2$ entries, and
$$
   A = \begin{pmatrix}
   yy^\t & O \\ O & yy^\t \end{pmatrix} ,
$$ 
where all blocks have size $n/2$. 
This matrix has rank two and 
its two nonzero eigenvalues are equal to $\|y\|_2^2$. Moreover,  
$A$ fulfills the equality $\cos(A,xx^\t) = 1/\sqrt{2}$,
thus showing that the strict inequality $c>1/\sqrt{2}$
in point 1 of Theorem \ref{thm:main}
can be necessary to have a simple eigenvalue.
Furthermore, the vector $v_1 = (y,0)^\t$ is an eigenvector
associated to the nontrivial eigenvalue,
and we have $\cos(v_1,x)^2 = \frac12$. That value meets
the lower bound in point 3 of the theorem.
Another counterexample in the same vein is given by 
$$
   A = \begin{pmatrix}
   O & yy^\t \\ yy^\t & O \end{pmatrix} .
$$ 
Here, the nonzero eigenvalues are simple
but have opposite sign and same modulus.


\section{The signature of a leading eigenvector}\label{sec:signature}

In this section we 
derive Theorem \ref{thm:tarazaga} as a special case of a much more general result,
using arguments very different from those in \cite{tarazaga2001}.
Our proof, which is founded on Theorem \ref{thm:HW},
extends immediately to eigenvectors having
a prescribed minimal number of nonnegative entries.

In what follows we denote by $\uno$ the all-ones vector
of appropriate size, and we let $J = \uno\uno^\t$ be an all-ones matrix. Moreover, we use the notation $\mathbb{P}^n_k$ 
to indicate the set of real $n$-vectors
having no more than $k$ (strictly) positive entries,
for $0 \leq k \leq n$.
In particular, 
$\mathbb{P}^n_n = \mathbb{R}^n$; 
$\mathbb{P}^n_{n-1}$ is the complement of the positive orthant
in $\mathbb{R}^n$; 
and $\mathbb{P}^n_{k_1}\subset \mathbb{P}^n_{k_2}$
for $k_1<k_2$.

\begin{lemma}   \label{lem:ckn}
For $0 < k \leq n$ it holds
$$
   \pi_{k,n} := 
   \max_{x\in\mathbb{P}^n_k}\cos(x,\uno) 
   = \sqrt{k/n} .
$$ 
Moreover, we have $\cos(x,\uno) = \pi_{k,n}$ if and only if
$x$ is any vector
with exactly $k$ entries 
having the same positive value
and the remaining $n-k$ entries being zeros.
\end{lemma}

\begin{proof}
Let $x\in \mathbb{P}^n_k$ and let $m$ be the number of its positive entries. Let $x^+$ be the positive $m$-dimensional vector made by the positive entries of $x$. 
By hypothesis $m \leq k$, therefore
\begin{equation}   \label{eq:cos}
   \cos(x,\uno) = \frac{x^\t \uno}{\sqrt{n}\|x\|_2}
   \leq \frac{\|x^+\|_1}{\sqrt{n}\|x^+\|_2}
   \leq \sqrt{k/n} .
\end{equation}
The rightmost inequality follows from the fact that,
for any vector $v\in\mathbb{R}^m$ one has
$\|v\|_1 \leq \sqrt{m}\|v\|_2$.
The last part of the claim is verified by 
requiring that both inequalities in \eqref{eq:cos}
hold as equalities.
\end{proof}

\begin{remark}   \label{rem:ckn}
By the preceding lemma, if a vector $x\in\mathbb{R}^n$ fulfills
$\cos(x,\uno) \geq \pi_{k,n}$ then either 
$\cos(x,\uno) = \pi_{k,n}$, whence $x$ is 
a special nonnegative vector having $k$ positive entries, 
or
$x\notin\mathbb P_k^n$, so it has at least $k+1$ positive entries. In both cases $x$ has at least $k+1$ nonnegative entries.
\end{remark}

\begin{theorem}   \label{thm:beyond}
Let $A$ be a symmetric $n\times n$ matrix such that
\begin{equation}   \label{eq:beyond}
   \uno^\t A\uno \geq \sqrt{(n-k)^2+k^2}\,\|A\|_\fro 
\end{equation}
for some $1\leq k< n/2$. Then, the spectral radius of $A$
is a simple eigenvalue, and a corresponding eigenvector can be oriented so that it has 
at least $n-k+1$ nonnegative entries, of which at least $n-k$ are positive.
\end{theorem}

\begin{proof}
Let $c = \cos(A,J)$.
By hypotheses and \eqref{eq:cosrk1},
$$
   c = \frac{\uno^\t A\uno}{n \|A\|_\fro}
   \geq \sqrt{\frac{(n-k)^2+k^2}{n^2}} .
$$
In particular, $c > 1/\sqrt{2}$.
From Theorem \ref{thm:main} we derive that 
the rightmost eigenvalue of $A$ is positive, simple and dominant,
therefore it coincides with $\rho(A)$.
Moreover, a corresponding eigenvector $v_1$ has
$\cos(v_1,\uno)^2 \geq \xi$ where
$$
   \xi = \frac12 + \frac12\sqrt{2c^2 - 1} \geq \frac{n-k}{n}
   = \pi_{n-k,n}^2.
$$   
By Remark \ref{rem:ckn}, $v_1$ must have at least $n-k+1$ nonnegative entries,
of which at least $n-k$ are positive, and the proof is complete.
\end{proof}

The hypothesis of the last theorem cannot be weakened, as shown by the following
construction. Consider the matrix 
$$
   A = \begin{pmatrix} J & O \\ O & J
   \end{pmatrix}
$$
whose diagonal blocks have order $k\times k$ and $(n-k)\times(n-k)$,
respectively. With this matrix the inequality \eqref{eq:beyond}
is fulfilled as equality. Moreover, an eigenvector 
corresponding to the largest eigenvalue of $A$ is 
$(0,\ldots,0,1,\ldots,1)^\t$, with exactly $n-k$ positive entries,
thus showing optimality of the claim.
Finally, we see immediately that Theorem \ref{thm:tarazaga}
is exactly the case $k = 1$ of the foregoing theorem.


\medskip
Any transformation $A\mapsto A + \alpha I$ leaves unchanged the 
eigenvectors of $A$ and translates its eigenvalues without affecting
their relative ordering. 
This fact suggests the next consequence.

\begin{corollary}
Let $A$ be a symmetric matrix such that for some $\alpha\in\mathbb{R}$ 
and $1\leq k< n/2$ we have
$$
   \uno^\t A\uno \geq \sqrt{(n-k)^2+k^2}\,\|A + \alpha I\|_\fro - n\alpha .
$$
Then the rightmost eigenvalue of $A$ is simple,
and a corresponding eigenvector can be oriented
so that it has at least $n-k+1$ nonnegative entries.
\end{corollary}

\begin{proof}
It is sufficient to apply Theorem \ref{thm:beyond}
to the matrix $B = A + \alpha I$. 
\end{proof}

A convenient, almost optimal value for $\alpha$ to be used in the last corollary 
is $\alpha = -\mathrm{tr}(A)/n$. In fact, with that value the
inequality simplifies considerably, as shown hereafter.

\begin{corollary}
Let $\mu$ and $\sigma^2$ be the mean and variance of the eigenvalues of $A$,
$$
   \mu = \frac1n \sum_{i=1}^n \lambda_i(A), \qquad
   \sigma^2 = \frac1n \sum_{i=1}^n (\lambda_i(A) - \mu)^2 .
$$
If for some $1\leq k< n/2$ we have
$$
   \frac1n \sum_{i\neq j}A_{ij} \geq 
   \sigma \sqrt{\frac{(n-k)^2+k^2}{n}}  
$$
then the rightmost eigenvalue of $A$ is simple,
and a corresponding eigenvector can be oriented
so that it has at least $n-k+1$ nonnegative entries.
\end{corollary}

\begin{proof}
Recall that $\mathrm{tr}(A) = \sum_i\lambda_i(A)$
and $\|A\|_\fro^2 = \sum_i\lambda_i(A)^2$. Now,
\begin{align*}
   \|A - \mu I\|_\fro^2 & = \|A\|_\fro^2 
   - 2\mu \sum_i A_{ii} + n\mu^2 \\
   & = \sum_i \lambda_i(A)^2 - 2\mu \sum_i\lambda_i(A) + n\mu^2 \\
   & = \sum_i (\lambda_i(A) - \mu)^2 = n\sigma^2 .
\end{align*}
The claim follows by rearranging terms and letting $\alpha = -\mu$ in the preceding corollary.
\end{proof}


\section{Spectral community detection in the stochastic block model}\label{sec:community}

The stochastic block model, also referred to as the planted partition model, is a popular generative model for random graphs
having a prescribed clustering structure, see e.g., \cite{Newman2012,SBM}.
It is often 
considered as a benchmark in the context of graph clustering and community detection on networks. 

Given a set of nodes $V$, the model 
in its most general form
assumes that a partition $\{C_1, \dots, C_k\}$ of $V$ is given, together with a set of connection probabilities $p_{ij}$,
for $1\leq i,j\leq k$. Each $p_{ij}$ defines the probability that any two vertices $u \in C_i$ and $v \in C_j$ are connected. For our purposes the interesting case is that of a bipartition $\{C, \bar C\}$, where $|C|=|\bar C|=n/2$, $\bar C = V\setminus C$,  and a pair  $p_{\mathrm {in}}, p_{\mathrm{out}}$ of connection probabilities: 
$p_{\mathrm {in}}$ is the probability that there exists an edge between any two nodes both belonging to the same subset, whereas $p_{\mathrm{out}}$ is the probability that there exists an edge between two nodes
belonging to different banks of the bipartition.

According to this model, $A$ is an $n\times n$ symmetric 
$\{0,1\}$-matrix whose entry $a_{ij}$ takes the value $1$ with probability 
$p_{\mathrm {in}}$ if both $i$ and $j$ belong to the same cluster, or $p_{\mathrm {out}}$ if $i$ and $j$ belong to different clusters. 
When $p_{\mathrm{in}}$ is bigger than $p_{\mathrm{out}}$ the edges tend to accumulate inside the clusters $C$ and $\bar C$. This phenomenon tends to set up $C$ and $\bar C$ as communities inside the network, and  the adjacency matrix of the graph tends to show a block diagonal predominance.

Various popular and effective techinques for revealing the community structure of a given graph or network 
with vertex set $V = \{1,\ldots,n\}$ are based on the spectral analysis of $M$, the modularity matrix of the graph. Several formulations of this matrix have been proposed in recent literature, see \cite{FT2015} for an overview. We consider here the one based on the Erd\"os-R\'enyi random graph model,  defined as follows:
$$
   M = A - \mathcal{P}_J (A) = A - \frac{\mathrm{vol} V}{n^2}J ,
$$
being $A$ the adjacency matrix of the given graph and $\mathrm{vol}V = \uno^\t A \uno$ its volume. 
The definition of this matrix is based on the following remark:
Let $\uno_S$ be the characteristic vector of the set $S \subseteq V$. Then,
$$
   \uno_S^\t M\uno_S = \uno_S^\t A \uno_S - 
   \frac{\mathrm{vol} V}{n^2}(\uno^\t\uno_S)^2 
   = \mathrm{vol}S - \mathrm{vol} V \frac{|S|^2}{n^2} , 
$$
where $\mathrm{vol}S = \uno_S^\t A \uno_S$ is the volume of $S$,
that is, the overall weight of internal edges, 
and the term $\mathrm{vol} V |S|^2/n^2$ quantifies the expected number of edges 
lying in $S$, if edges where placed uniformly at random in the graph. Hence, the maximization of the Rayleigh quotient
$q(S) = \uno_S^\t M\uno_S/\uno_S^\t \uno_S$ with respect to $S$ 
arises naturally as a theoretically based method to detect
the presence of a cluster in the graph: 
If $q(S)$ is ``large'' then the edge density in the subgraph induced by $S$ 
is higher than expected.
However, the combinatorial nature of the maximization procedure
makes it an NP-hard problem. Effective algorithms can be based on a continuous relaxation, leading to the so-called spectral methods.

Loosely speaking, spectral methods for community detection locate clusters in a given graph according to the sign of the entries of the leading eigenvector of $M$. 
The idea, firstly introduced in physics literature \cite{NewmanGirvan,Newman2006},
mirrors the analogous approach widely used for solving graph partitioning problems by exploiting spectral properties of Laplacian matrices, as pioneered by Fiedler \cite{Fiedler1,Fiedler2}.

Given the bipartition $\{C, \bar C\}$ of the vertex set $V = \{1,\ldots,n\}$, 
define the vector
$z = (z_1,\ldots,z_n)^\t$ where $z_i = 1$ if $i\in C$ and
$z_i = -1$ if not, and let $Z = zz^\t$.
By using Theorem \ref{thm:main} we can show that, for certain values of  $p_{\mathrm{in}}$ and $p_{\mathrm{out}}$, the value of
$\cos(M,Z)$ is large, and
the leading eigenvector of $M$ is almost parallel to $z$. 
Consequently, partitioning vertices on the basis of the 
sign of the corresponding entries of the leading eigenvector yields a reliable approximation
of the true bipartition.

Since $z^\t\uno = 0$ we have 
$$
   \Tr(MZ)= z^\t Mz = z^\t Az - \frac{\uno^\t A \uno}{n^2}z^\t Jz 
   = z^\t Az . 
$$   
Moreover, $a_{ij} \in \{0,1\}$ implies $a_{ij}^2=a_{ij}$, 
therefore $\|A\|_\fro^2 = \uno^\t A\uno$ and
$$ 
   \|M\|_\fro^2 = \|A\|_\fro^2 -\frac{(\uno^\t A \uno)^2}{n^2}
   =(\uno^\t A \uno)\bigg( 1-\frac{\uno^\t A \uno}{n^2} \bigg)\, .
$$
As the entries of $A$ are independent Bernoulli random variables, the quantities $\uno^\t A \uno$ and $z^\t A z$ are independent random variables as well. In fact, let $\nu_C=\uno_C^\t A \uno_C+\uno_{\bar C}^\t A \uno_{\bar C}$ and $\partial_C=2(\uno_C^\t A\uno_{\bar C})$. Then we have $\Tr(MZ)=\nu_C-\partial_C$ and $\uno^\t A \uno = \nu_C+\partial_C$. Both $\nu_C$ and $\partial_C$ are the sum of i.i.d.\ Bernoulli trials, thus they follow a binomial distribution:
$$
   \nu_C \sim B(n^2/2,p_{\mathrm{in}}) , \qquad
   \partial_C \sim B(n^2/2,p_{\mathrm{out}}) .
$$   
Consequently, 
we have the following statistics:
\begin{align*}
   \mathbb{E}(\uno^\t A\uno) & = 
   \mathbb{E}(\nu_C) + \mathbb{E}(\partial_C) = 
   n^2\,\frac{p_{\mathrm{in}}+p_{\mathrm{out}}}{2} \\
      \mathbb{E}(z^\t Az) & = 
   \mathbb{E}(\nu_C) - \mathbb{E}(\partial_C) = 
   n^2\,\frac{p_{\mathrm{in}}-p_{\mathrm{out}}}{2} \\
   \mathrm{Var}(\uno^\t A\uno) & = 
   \mathrm{Var}(\nu_C) + \mathrm{Var}(\partial_C) =
   \frac{n^2}{2} \big(p_{\mathrm{in}}(1-p_{\mathrm{in}}) 
   + p_{\mathrm{out}}(1-p_{\mathrm{out}}) \big) .
\end{align*}
Moreover, we also have 
$\mathrm{Var}(z^\t Az) = \mathrm{Var}(\uno^\t A\uno)$. 
With the help of the foregoing formulas, the subsequent
lemma estimates the average value of
$\cos(M,Z)$ when $M$ is the modularity matrix of a graph belonging to the stochastic block model introduced before.

\begin{lemma}   \label{lem:stochastic-block-model}
 Let $M = A -(\mathrm{vol} V/n^2)J$ be the Erd\"os-R\'enyi modularity matrix of a  graph belonging to the stochastic block model with two equally sized clusters $C$ and $\bar C$ and  edge probabilities $p_{\mathrm{in}}$ and $p_{\mathrm{out}}$. 
Let $z = (z_1,\ldots,z_n)^\t$ where $z_i = 1$ if $i\in C$ and
$z_i = -1$ if not. Moreover, let
$$
   \gamma = \frac{p_{\mathrm{in}}-p_{\mathrm{out}}}
   {\sqrt{(p_{\mathrm{in}}+p_{\mathrm{out}})
   (2-p_{\mathrm{in}}-p_{\mathrm{out}})}} .
$$
For any fixed $\varepsilon > 0$,
with probability converging to $1$ as $n\to \infty$
we have $\cos(M,Z)^2 \geq \gamma^2 - \varepsilon$.
\end{lemma}

\begin{proof}
Since $z^\t z = n$, we have the equivalent formulas
$$
   \cos(M,Z) = \frac{z^\t Mz}{n\|M\|_\fro}
    = \frac{z^\t Az}{n^2} \, \frac{n}{\|M\|_\fro} .
$$
Hence, the inequality $\cos(M,zz^\t)^2 \geq \gamma^2 - \varepsilon$
is equivalent to the condition
\begin{equation}   \label{eq:cond}
   \frac{(z^\t Az)^2}{n^4}   
   \geq (\gamma^2 - \varepsilon) \frac{\|M\|_\fro^2}{n^2} .
\end{equation}
From the equation $\mathbb{E}((\uno^\t A\uno)^2) =
\mathrm{Var}(\uno^\t A\uno) + \mathbb{E}(\uno^\t A\uno)^2$
and the preceding expressions 
we can obtain the expectation of $\|M\|_\fro^2/n^2$ in 
the considered stochastic block model:
\begin{align*}
   \frac{1}{n^2}\mathbb{E}\big(\|M\|_\fro^2) & = 
   \frac{1}{n^2}\mathbb{E}(\uno^\t A \uno) - \frac 1 {n^4} 
   \big( \mathrm{Var}(\uno^\t A \uno) + 
   \mathbb{E}((\uno^\t A \uno))^2 \big) \\
   & = \frac{p_{\mathrm{in}}+p_{\mathrm{out}}}{2}
    - \frac 1 {n^4} \bigg( \frac{n^2(p_{\mathrm{in}}-p_{\mathrm{in}}^2 
   + p_{\mathrm{out}}-p_{\mathrm{out}}^2)}{2}
   + \frac{n^4(p_{\mathrm{in}}+ p_{\mathrm{out}})^2}{4} \bigg) \\
   & = \frac{(p_{\mathrm{in}}+p_{\mathrm{out}})
   (2-p_{\mathrm{in}}-p_{\mathrm{out}})}{4}
    + \mathcal{O}(n^{-2}) .
\end{align*}
Since $\mathrm{Var}(z^\t Az) = \mathrm{Var}(\uno^\t A\uno)$, we also have 
\begin{align*}
   \frac{1}{n^4} \mathbb{E}\big((z^\t Az)^2) & = 
   \frac{1}{n^4} \big(
   \mathrm{Var}(z^\t Az) + \mathbb{E}(z^\t Az)^2 \big) \\
   & = \frac{(p_{\mathrm{in}}- p_{\mathrm{out}})^2}{4} 
   + \mathcal{O}(n^{-2}).
\end{align*}
Owing to the independence of $\|M\|_\fro^2$
and $(z^\t Az)^2$, with the value of $\gamma$ given in the claim
the inequality \eqref{eq:cond}
is certainly fulfilled in the limit of large $n$.
\end{proof}

Finally, we exploit Theorem \ref{thm:main}
and the preceding lemma
to estimate certain spectral properties of the modularity matrix $M$ that are relevant
for the community detection problem. In particular, we focus on $\lambda_1(M)$,
which is an important indicator of the detectability of the planted partition
$\{C,\bar C\}$ \cite{Newman2012,Newman2006}, and the overlap between that
partition and the nodal sets of a leading eigenvector of $M$.

\begin{theorem}   \label{thm:stochastic-block-model}
In the same hypotheses and notations of Lemma \ref{lem:stochastic-block-model},
with probability approaching $1$ in the limit for large $n$ we have:
\begin{enumerate}
\item
$\lambda_1(M) \geq \mu$ and $|\lambda_1(M) - \mu|/\lambda_1(M) \leq \sqrt{1 - \gamma^2}$
where $\mu = (p_{\mathrm{in}}-p_{\mathrm{out}})n/2$.
\item 
Let $v_1$ be an eigenvector of $\lambda_1(M)$. 
If $\gamma^2 \geq \frac12$ then $\cos(v_1,z)^2 \geq \bar \xi$ where
$$
   \bar \xi = \frac12 + \frac12\sqrt{2\gamma^2 - 1} .
$$
\item
The fraction of vertices classified correctly by the
partition $\{i : (v_1)_i \geq 0\}$ and $\{i : (v_1)_i < 0\}$
is at least $\bar\xi$.
\end{enumerate}
\end{theorem}

\begin{proof}
The first two claims are easy consequences of Lemma \ref{lem:stochastic-block-model} and Theorem \ref{thm:main}.
In fact, the value of $\mu$ comes from the average value
of $z^\t Mz/z^\t z$, which is equal to $z^\t Az/n$, while $\bar\xi$ is the largest root of $\gamma^2 =2\xi^2-2\xi+1$.

Let $v_1$ be an eigenvector of $\lambda_1(M)$ oriented so that $\cos(v_1,z) > 0$.
Define $w\in\mathbb{R}^n$ as $w_i = 1$ if $v_i \geq 0$ and $-1$ if not.
Clearly, $0 < \cos(v_1,z) \leq \cos(w,z) = \cos(w\circ z,\uno)$ where
$w\circ z$ is the Hadamard (componentwise) product of $w$ and $z$.
Note that the number of positive entries in $w\circ z$ gives exactly the number
of correctly classified vertices. Let $k$ be the integer part of $\bar \xi n$,
so that we have $\sqrt{\bar \xi} \geq \sqrt{k/n} = \pi_{k,n}$.
Hence, $\cos(w\circ z,\uno) \geq \cos(v_1,z) \geq \pi_{k,n}$.
The last claim follows from Remark \ref{rem:ckn}.
\end{proof}




{\small
\noindent{\em Authors' addresses}:

{\em Dario Fasino}, Dipartimento di Scienze matematiche, informatiche e fisiche,
Universit\`a di Udine, Udine, Italy,
e-mail: \texttt{dario.fasino@uniud.it}.
 
{\em Francesco Tudisco}, Department of Mathematics and Computer Science, Saarland University, Saarbr\"ucken, Germany,
e-mail: \texttt{tudisco@cs.uni-saarland.de}.

}

\end{document}